\documentclass[smallextended,draft,onecolumn]{svjour3}

\usepackage{graphicx}
\usepackage{mathptmx}
\usepackage{latexsym}
\usepackage{amssymb,amsfonts,amsmath}

\journalname{Optim lett}

\begin{document}

\title{Optimal sequencing of a set of positive numbers with the variance of the sequence's partial sums maximized }
\titlerunning{Optimal sequencing of a set of positive numbers}

\author{Li Wei  \and  Wangdong Qi  \and Dingxing Chen   \and Peng Liu  \and En Yuan}
\institute{Li Wei \at
              PLA University of \,Science and Technology\,,
              Nanjing 210007, China\\
              \email{wlnb@hotmail.com}
           \and
           Wangdong Qi , Dingxing Chen , Peng Liu , En Yuan \at
              PLA University of \,Science and Technology\,,
              Nanjing 210007, China
}

\date{Received: date / Accepted: date}

\maketitle

\begin{abstract}
   We consider the problem of sequencing a set of positive numbers. We try to find the optimal sequence to maximize the variance of its partial sums. The optimal sequence is shown to have a beautiful structure. It is interesting to note that the symmetric problem which aims at minimizing the variance of the same partial sums is proved to be NP-complete in the literature.
\end{abstract}

\section{Introduction}
  \label{sec:Introduction}
   This paper considers the problem of sequencing a set of positive numbers to obtain a sequence with the variance of its partial sums maximized (named $\mathrm{VPS_{max}}$ problem for simplicity). Before giving formal formulation of the problem, we introduce some denotations.

   Let \,$\vec{A}=\left\{a_k,1\le k\le n\mid 0<a_1<a_2<\dots<a_n\right\}$\, be a set of  positive numbers sorted in ascending order and $\vec{C}=\pi(\vec{A})=\left\{c_k,1\le k\le n\right\}$\,be a sequencing of the elements \,$a_1,a_2,\dots,a_n $. We define the partial sums of sequence \,$\vec{C}$\, as\, $\vec{S}=\left\{s_k,\,1\le k\le n\right\}$ with
        $$ s_k= \sum_{i=1}^k c_i ,1\le k\le n$$

   The mean $ \overline{\vec{S}} $ and variance $ V(\vec{S}) $ of $ \vec{S} $ are
   $$ \overline{\vec{S}}=\frac{1}{n} \sum_{k=1}^n s_k ,V(\vec{S})=\frac{1}{n}\sum_{k=1}^n (s_k-\overline{\vec{S}})^2$$

   Since $ V(\vec{S}) $ is also a function of $\vec{C}$, it can be denoted as $f(\vec{C})$
   $$V(\vec{S})=\frac{1}{n}\sum_{k=1}^n \Bigl(\sum_{i=1}^k c_i-\frac{1}{n}\sum_{k=1}^n\sum_{i=1}^k c_i  \Bigr)^2=f(\vec{C}) $$

   The optimal solution to the $\mathrm{VPS_{max}}$ problem is defined as a sequence $ \vec{C}^\ast $, formed by sequencing the elements of $\vec{A}$, satisfying
    $$ \vec{C}^\ast=\arg\max_{\vec{C}}f(\vec{C})=\arg\max_{\vec{C}}V(\vec{S}) $$
   That is to say, an optimal solution is a sequence with the variance of its partial sums maximized.

   The $\mathrm{VPS_{max}}$ problem arises from the optimization of the measurement frequencies in a new kind of radio interferometry \cite{Maroti-05}\,, which is a promising ranging technique in wireless sensor networks. To the best of our knowledge, $\mathrm{VPS_{max}}$ has not been explored until now.

   It should be pointed out that the completion time variance (CTV) problem, which in essence aims at sequencing a set of positive numbers to minimize the variance of the sequence's partial sums, is a symmetric problem of $\mathrm{VPS_{max}}$. It has been studied extensively by the operations research community for decades \cite{Merten-72}-\cite{Ye-07}.

   In sharp contrast to CTV, which is found to be NP-complete in the literature \cite{Kubiak-93}, it is found in this paper that the optimal solution to $\mathrm{VPS_{max}}$ has a very nice structure.

\section{Preliminaries}
   \label{sec:Preliminaries}
   \setlength{\parskip}{0.5ex}   

    \begin{definition}
    \label{def:muij}
      The mean of the elements in $ \vec{S}=\left\{s_k,1\le k\le n \right\} $ ranging from the $i$th element to the $(j-1)$th element is defined as the\, $(i,j)$\begin{it}-partial mean\end{it} \,of $\vec{S}$, denoted as
       $$\mu_{ij}(\vec{S})=\frac{1}{j-i}\sum_{k=i}^{j-1}s_k,\,1\le i<j\le n+1$$
    \end{definition}

   Based on the monotonically increasing property of sequence \,$ \vec{S}=\left\{s_k,1\le k\le n \right\} $, we show that \,$\mu_{ij}(\vec{S})$
    \,has the following properties:
   \renewcommand{\labelenumi}{(\arabic{enumi})}   
   \begin{enumerate}
     \renewcommand{\baselinestretch}{1.4} \normalsize
      \item	 $\overline{\vec{S}}=\mu_{1n+1}(\vec{S})$;
      \item	 $\mu_{ij}(\vec{S})$ is a strictly monotonically increasing function of \,$i$ \,and\, $j$;
      \item for all $1\le i<j<k\le n+1,\mspace{4mu}
            \frac{j-i}{k-i}\,\mu_{ij}(\vec{S})+\frac{k-j}{k-i}\,\mu_{jk}(\vec{S})=\mu_{ik}(\vec{S}) $;
      \item $\mu_{ij}(\vec{S})<\mu_{kl}(\vec{S}) $ \,for\, $1\le i<k<l\le n+1$  \,and\, $1\le i<j<l\le n+1$ .
   \end{enumerate}

    To find the sequence that maximize the variance of its partial sums, we try to start from an arbitrary sequencing of $a_1,a_2,\dots,a_n $ and follow a path of favorable transforms that would eventually lead to the optimal one. We begin with the simplest transform from one sequence to another by interchanging positions of only two elements.

    \begin{definition}
    \label{def:ijchange}
       The\, transform\, from\, the\, sequence\, $ \vec{C}=\{c_1,\dots,c_{i-1},c_i,c_{i+1},\dots,c_{j-1},c_j,c_{j+1},\dots,c_n\} $ \, to\,  $\vec{C}^{\prime}=\{c_1,\dots,c_{i-1},c_j,c_{i+1},\dots,c_{j-1},c_i,c_{j+1},\dots,c_n\} $  by interchanging the $i$th element and the $j$th element is called an\,  $(i,j)$\textit{-interchange}\, of \,$ \vec{C} $.
    \end{definition}\par
   \vspace{-8pt}
   An $(i,j)$\textit{-interchange}\, from \,$\vec{C}$\, to\, $\vec{C}^{\prime}$ \,is a\, \textit{favorable\, transform} \, and \,$\vec{C}^{\prime}$ \,is called a\, \textit{better} sequence \,if \,$f(\vec{C}^{\prime})>f(\vec{C})$.\par
   Next, we give a criterion to determine whether a given\, $(i,j)$\textit{-interchange}\, is\, a\, \textit{favorable}\, transform.\par

   Let \,$ \vec{C}^{\prime} $\, denotes the transform of \,$ \vec{C}=\left\{c_k,1\le k\le n\right\} $\, via \,$(i,j)$\textit{-interchange}\, and \,$ \delta=c_j-c_i$. We get
   \begin{equation*}s_k^{\prime}=
    \begin {cases} s_k+\delta & i\le k<j\\s_k & others
    \end {cases}
   \end{equation*}
   and
      $$ \overline{\vec{S}^{\prime}}=\overline{\vec{S}}+\frac{(j-i)\delta}{n} $$
   Since the objective function can be simplified as
   \begin{align}
   f(\vec{C})=V(\vec{S})=\frac{1}{n}\sum_{k=1}^n s_k^2-\overline{\vec{S}}^2
   \end{align}
   It follows that
    \begin{align*}
    f(\vec{C}^{\prime})-f(\vec{C})&=\Biggl( \frac{1}{n}\sum_{k=1}^n {s_k^{\prime}}^2-\overline{\vec{S}^{\prime}}^2  \Biggr)-\Biggl( \frac{1}{n}\sum_{k=1}^n s_k^2-\overline{\vec{S}}^2  \Biggr)\\
     &= \left(\frac{1}{n}\sum_{k=i}^{j-1} \left( (s_k+\delta)^2-s_k^2\right) \right)-
       \left( \left(\overline{\vec{S}}+\frac{(j-i)\delta}{n} \right)^2-\overline{\vec{S}}^2 \right)\\
     &= \left(\frac{2\delta}{n}\sum_{k=i}^{j-1}s_k+ \frac{(j-i)}{n}\delta^2 \right) -
       \left( \frac{2(j-i)\delta}{n}\overline{\vec{S}} +\frac{(j-i)^2\delta^2}{n^2}  \right)\\
     &= \frac{2(j-i)\delta}{n} \left(\frac{1}{j-i}\sum_{k=i}^{j-1}s_k -\overline{\vec{S}} +
        \frac{\delta}{2}\left( 1 -\frac{(j-i)}{n}\right)  \right)
    \end{align*}

   \noindent In latter parts of this paper, we denote the difference \,$ f(\vec{C}^{\prime})-f(\vec{C}) $\, as \,$ \Delta {f}(i,j,\vec{C},\vec{C}^{\prime}) $\,. Then

   \begin{equation}
   \label{equ:deltf}
    \Delta{f}(i,j,\vec{C},\vec{C}^{\prime})= \frac{2(j-i)\delta}{n} \left( \mu_{ij}(\vec{S})-\overline{\vec{S}} +
        \frac{\delta}{2}\left( 1 -\frac{(j-i)}{n}\right)  \right)
   \end{equation}

   Let \,$\varphi(\delta)=D_1\delta^2+D_2\delta=\Delta f(i,j,\vec{C},\vec{C}^{\prime}) $, where \,$D_1=\frac{j-i}{n}\left(1-\frac{j-i}{n} \right)>0$ \,and \,$D_2=\frac{2(j-i)}{n}\left( \mu_{ij}(\vec{S})-\overline{\vec{S}}\right)$.
   Based on the property of quadratic equation, we have criterions for\, \textit{favorable} \,$(i,j)$\textit{-interchange}.
   \begin{itemize}
     \item Claim 1: if \,$D_2>0$\,  $\left( \textrm{i.e.}\mspace{6mu}\mu_{ij}(\vec{S})>\overline{\vec{S}}\,\right)$, \,then\, the \,$(i,j)$\textit{-interchange} \,is\, \textit{favorable}\, when\, $\delta>0$  \,or\,  $\delta<-D_2/D_1$.
         \vspace{0pt}
     \item Claim 2: if $D_2\le0$  $\left( \textrm{i.e.} \mspace{6mu}\mu_{ij}(\vec{S})\le\overline{\vec{S}}\,\right)$, \,then\, the $(i,j)$\textit{-interchange} \,is\, \textit{favorable} \,when\, $\delta<0$  \,or\,  $\delta>-D_2/D_1$.
   \end{itemize}

    We then have the following results.
    	
    \spnewtheorem{propo}{Proposition}{\bf}{\it}
    \begin{propo}
    \label{propo:c1}
     If $c_1\ne a_1$, \,then sequence $\vec{C}$  is not the optimal solution.
    \end{propo}

    \begin{proof}
      {Clearly, \,$\mu_{1j}(\vec{S})<\mu_{1n+1}(\vec{S})=\overline{\vec{S}} $, for any integer $j \in [2,n]$. Therefore, if $a_1=c_j$, we will obtain a\, \textit{better}\, sequence than\, $\vec{C}$\, via\, $(1,j)$\textit{-interchange}\, according to claim 2.} \qed  	
    \end{proof}

    \begin{propo}
    \label{propo:cn}
       If  $c_n=a_n$ with $n>3$, \,then sequence $\vec{C}$ is not the optimal solution.
    \end{propo}

    \begin{proof}
    Note that
    \begin{equation*}
         \mu_{n-1n}(\vec{S})=\sum_{k=1}^{n-1}c_k,\,\overline{\vec{S}}=\sum_{k=1}^{n}\frac{n-k+1}{n}c_k,\,\delta=c_n-c_{n-1}>0
    \end{equation*}
    and
      \begin{align*}
         \hspace{-0pt}\Delta{f}(n-1,n,\vec{C},\vec{C}^{\prime})&=\frac{2\delta}{n} \left(\mu_{n-1n}(\vec{S})-\overline{\vec{S}}+\frac{\delta}{2}(\frac{n-1}{n}) \right)\\
         &=\frac{2\delta}{n} \left(\sum_{k=1}^{n-1}\frac{k-1}{n}c_k-\frac{c_n}{n}+\frac{(n-1)(c_n-c_{n-1})}{2n} \right)\\
         &=\frac{2\delta}{n} \left(\sum_{k=1}^{n-2}\frac{k-1}{n}c_k +\frac{n-2}{n}c_{n-1} -\frac{c_n}{n}+\frac{(n-1)(c_n-c_{n-1})}{2n} \right)\\
         &=\frac{2(c_n-c_{n-1})}{n} \left(\sum_{k=1}^{n-2}\frac{k-1}{n}c_k +\frac{(n-3)(c_n+c_{n-1})}{2n} \right)
    \end{align*}
    The expression is always positive for $n>3$, i.e. $\Delta{f}(n-1,n,\vec{C},\vec{C}^{\prime})>0$. So, $(n-1,n)$\textit{-interchange} will obtain a\, \textit{better} \,sequence than $\vec{C}$. \qed
    \end{proof}

    \begin{propo}
    \label{propo:mean}
      If sequence $\vec{C}$ is the optimal solution with \,$c_k=a_n,1<k<n$, then
      \renewcommand{\labelenumi}{\it{(\arabic{enumi})}}
       \begin{enumerate}
        \renewcommand{\baselinestretch}{1.2}\normalsize
         \item	for \,any \,$i,j$ \,with \,$1\le i<j\le k$, \,$\mu_{ij}(\vec{S})<\overline{\vec{S}}$;
         \item	for \,any \,$i,j$ \,with \,$k\le i<j\le n$, \,$\mu_{ij}(\vec{S})>\overline{\vec{S}}$.
       \end{enumerate}
    \end{propo}

    \begin{proof}
    (1) Otherwise, there exist \,$i,j$\, such that \,$1\le i<j\le k$ \,and\, $\mu_{ij}(\vec{S})\ge \overline{\vec{S}}$, \,then
     $$ \mu_{ik}(\vec{S})\ge \mu_{ij}(\vec{S})\ge \overline{\vec{S}} $$
     A \,\textit{better}\, sequence will be created when exchanging the \,$i$th\, and the \,$k$th\, element of \,$\vec{C}$\, due to \,$\delta=c_k-c_i>0$. This conflicts with the optimality of \,$\vec{C}$.\par
     A similar proof can be applied to (2). \qed
    \end{proof}

    \begin{definition}
    \label{def:shape}
       A sequence is called a \,\begin{it}$\wedge$-shaped\end{it} \,sequence when the elements before the largest one are sorted in an ascending order, while the elements after the largest one are sorted in a descending order.
    \end{definition}

    \begin{propo}
       \label{propo:shape}
         The optimal sequence is $\wedge$-Shaped. In other words, if sequence $\vec{C}$ is the optimal solution with $c_k=a_n$, $1<k<n$, then
         \renewcommand{\labelenumi}{\it{(\arabic{enumi})}}
         \begin{enumerate}
           \renewcommand{\baselinestretch}{1.2}\normalsize
           \item	  for any \,$i,j$ \,with\, $1\le i<j\le k$, \,$c_i<c_j$;
           \item	  for any \,$i,j$ \,with\, $k\le i<j\le n$, \,$c_i>c_j$.
         \end{enumerate}
    \end{propo}

    \begin{proof}
    It is a direct conclusion of Proposition \ref{propo:mean}. \qed
    \end{proof}
    \noindent\textbf{Remark.} There exists similar property for the CTV problem. Eilon and Chowdhury \cite{Eilon-77} proved that the optimal solutions to the CTV minimization problem should be V-shaped, meaning that the elements before the smallest one are sorted in a decreasing order, while the elements after the smallest one are sorted in an ascending order.

\section{Main results}
   \label{sec:Main results}
    \begin{definition}
    \label{def:dual}
      For a sequence $\vec{C}=\{c_k,1\le k\le n\}$, do $(k,n+2-k)$\textit{-interchange} for all $2\le k\le u$, and we can obtain the \,\begin{it}dual sequence\end{it} \,$\vec{C}^d$\, of\, $\vec{C}$, where $u=\bigl\lceil{n/2}\bigr\rceil$,  \,$\lceil\,\rceil$ denotes the ceiling function.
    \end{definition}

    \begin{lemma}
       \label{lemma:dual}
           For sequence $\vec{C}=\{c_1,c_2,c_3,\dots,c_{n-1},c_n\}$ and its dual sequence $\vec{C}^d=\{c_1,c_n,c_{n-1},\dots,c_3,c_2\}$, \,the following identity holds.
           $$f(\vec{C})=f(\vec{C}^d)$$
    \end{lemma}

   \begin{proof}
    See \cite{Merten-72}, we also provide another proof of Lemma \ref{lemma:dual} in the appendix.\qed
    \end{proof}

    \begin{lemma}\label{lemma:kinds}
           There exist at least two kinds of optimal sequences, one is in the form of \,$c_1=a_1,c_2=a_2$, \,and the other is in the form of \,$c_1=a_1,c_n=a_2$.
    \end{lemma}

   \begin{proof}
   We consider the position of \,$a_n$\, in the optimal sequence first. Since $c_1=a_1$ and $c_n\ne a_n$, by Proposition \ref{propo:shape}, we have
     \vspace{-4pt}
     \begin{enumerate}
       \renewcommand{\baselinestretch}{1.2}\normalsize
       \item	case \,$c_2=a_n$: then \,$c_n=a_2$;
       \item	case \,$c_k=a_n,2<k<n$: then\, $c_n=a_2$ \,or \,$c_2=a_2$;
     \end{enumerate}
     \par\vspace{-4pt}
   This means either \,$c_n=a_2$ \,or \,$c_2=a_2$. When \,$c_n=a_2$, $c_2=a_2$\, is also the optimal position for \,$a_2$\, by Lemma \ref{lemma:dual} and vice versa. Therefore, both \,$c_n=a_2$ \,and\, $c_2=a_2$ \,are the optimal position for \,$a_2$. \qed
   \end{proof}
   \vspace{0pt}
   \noindent\textbf{Remark.} Note that the optimal sequence of the CTV problem is in the form of \,$c_1=a_n,c_2=a_{n-1}$ \,or\, $c_1=a_n$, $c_n=a_{n-1}$ \cite{Schrage-75} \cite{Hall-91}. Similarity appears once again. The major difference will be shown below.

   \begin{definition}
    \label{def:sumn+2}
     For a sequence $\vec{C}=\{c_k,1\le k\le n\}$, \,if the following two conditions do not hold,
         \begin{enumerate}
          \renewcommand{\baselinestretch}{1.2}\normalsize
          \item	  $c_k<c_{n+2-k}$ , for all $2\le k\le u$ ;
          \item	  $c_k>c_{n+2-k}$ , for all $2\le k\le u$ .
        \end{enumerate}\par
      \noindent then do $(k,n+2-k)$\textit{-interchange} for each \,$k$\, satisfying \,$c_k<c_{n+2-k},2\le k\le u$\, until all \,$c_k>c_{n+2-k},2\le k\le u$. This is called the \,\textit{sum-`$n$+2' transform} of \,$\vec{C}$.
   \end{definition}
    \begin{theorem}
    \label{theo:sumn+2}
     A sum-`n+2' transform always results in a \textit{better} sequence.
    \end{theorem}

   \begin{proof}
   Note that
       $$ \mu_{ij}(\vec{S})=\sum_{k=1}^{i-1}c_k+\sum_{k=i}^{j-1}\frac{j-k}{j-i}c_k$$
   We begin by rewriting (\ref{equ:deltf}) as:
     \begin{align}
       \label{equ:dfini}
        \Delta{f}(i,j,\vec{C},\vec{C}^{\prime})&= \frac{2(j-i)(c_j-c_i)}{n} \Biggl( \sum_{k=1}^{i-1}c_k+\sum_{k=i}^{j-1}\frac{j-k}{j-i}c_k -\sum_{k=1}^{n}\frac{n+1-k}{n}c_k +\frac{c_j-c_i}{2}\left( 1 -\frac{j-i}{n}\right)  \Biggr) \notag \\
        &= \frac{2(j-i)(c_j-c_i)}{n^2}\Biggl( \sum_{k=1}^{i-1}(k-1)c_k+ \frac{1}{j-i}\sum_{k=i+1}^{j-1}\left((j-i-n)k-j+i+ni\right)c_k\notag \\
         & \hspace{85pt}-\sum_{k=j+1}^{n}(n+1-k)c_k +( i+j-n-2)(c_j+c_i)/{2} \Biggr )
     \end{align}
  If $i+j=n+2$, we obtain
    \begin{align}
       \Delta{f}(i,n+2-i,\vec{C},\vec{C}^{\prime})=&\frac{2(n+2-2i)(c_{n+2-i}-c_i)}{n^2} \Biggl( \sum_{k=2}^{i-1}(k-1)(c_k-c_{n+2-k})+\frac{i-1}{n+2-2i}\sum_{k=i+1}^{n+1-i}(n+2-2k)c_k\Biggr) \notag \\
                                      =& \frac{2(n+2-2i)(c_{n+2-i}-c_i)}{n^2} \Biggl( \sum_{k=2}^{i-1}(k-1)(c_k-c_{n+2-k})\notag \\
                                        &\hspace{115pt}+\frac{i-1}{n+2-2i}\sum_{k=i+1}^{u}(n+2-2k)(c_k-c_{n+2-k})\Biggr) \notag \\
                                      =&\frac{2(c_{n+2-i}-c_i)}{n^2} \Biggl( (n+2-2i)\sum_{k=2}^{i-1}(k-1)(c_k-c_{n+2-k})\notag \\
                                        &\hspace{65pt}+(i-1)\sum_{k=i+1}^{u}(n+2-2k)(c_k-c_{n+2-k})\Biggr)   
    \end{align}
   where $u=\bigl\lceil{n/2}\bigr\rceil$.\par
   Without loss of generality, Let $I=\left\{i_1,i_2,\dots,i_p\mid 2\le i_1<i_2\dots <i_p\le u \right\}$ denote the aggregate set of the indexes \,$k$\, in \,$\vec{C}=\{c_k,k=1\dots n\}$\, satisfying\, $c_{i_q}<c_{n+2-i_q}, q=1,2\dots p$\,.\par
   We define \,$\vec{C}^{(m)}=\{c_k^{(m)},k=1\dots n,m\le p \}$\, as the transform of sequence \,$\vec{C}$ \,via\, $(i_1,n+2-i_1),(i_2,n+2-i_2)\dots (i_m,n+2-i_m)$\textit{-interchange}\,, which is also denoted as \,$\vec{C}^{(0)}$\, for \,notational convenience. By definition, $\vec{C}^{(m)}$\, may also be obtained from \,$\vec{C}^{(m-1)}$\, by \,$(i_m,n+2-i_m)$ \textit{-interchange}. The element \,$c_k^{(m)}$\, is given by
      \begin{align}
        \label{equ:ck}
            c_k^{(m)}&=
                  \begin {cases} c_{n+2-k}^{(m-1)} &  k=i_m\\
                                 c_{n+2-k}^{(m-1)} &  k=n+2-i_m\\
                                     c_{k}^{(m-1)} &  others
                  \end {cases}\notag\\
                     &=
                 \begin {cases} c_{n+2-k} &  k=i_q,1\le q\le m\\
                                c_{n+2-k} &  k=n+2-i_q,1\le q\le m\\
                                    c_{k} &  others
                 \end {cases}
      \end{align}
   We then have
      \begin{align*}
      f\left(\vec{C}^{(m)}\right)-f\left(\vec{C}^{(m-1)}\right)&=\Delta f\left(i_m,n+2-i_m,\vec{C}^{(m-1)},\vec{C}^{(m)} \right)\\
                                                   &=\frac{2}{n^2}\left(c_{n+2-i_m}^{(m-1)}-c_{i_m}^{(m-1)}\right) \Biggl(
                                                     (n+2-2i_m)\sum_{k=2}^{i_m-1}(k-1)\left(c_k^{(m-1)}-c_{n+2-k}^{(m-1)} \right) \\
                                                   &\hspace{110pt} +(i_m-1)\sum_{k=i_m+1}^{u}(n+2-2k)\left(c_k^{(m-1)}-c_{n+2-k}^{(m-1)} \right)\Biggr)
      \end{align*}
   It follows from (\ref{equ:ck}) that
      \begin{align}
      \Delta f(i_m,n+2-i_m,&\vec{C}^{(m-1)},\vec{C}^{(m)})=\frac{2}{n^2}\left(c_{n+2-i_m}-c_{i_m}\right)
      \Biggl[ (n+2-2i_m) \notag\\
                         &\hspace{0pt} \biggl( \sum_{k\in I,k<i_m}(k-1)\left(c_k^{(m-1)}-c_{n+2-k}^{(m-1)}\right)
                                +\sum_{k\not\in I,2\le k<i_m}(k-1)\left(c_k^{(m-1)}-c_{n+2-k}^{(m-1)}\right) \biggr)\notag\\  &\hspace{60pt} +(i_m-1)\sum_{k=i_m+1}^{u}(n+2-2k)(c_k-c_{n+2-k})\Biggr] \notag\\
                  =&\frac{2}{n^2}\left(c_{n+2-i_m}-c_{i_m}\right) \Biggl[ (n+2-2i_m)\biggl( \sum_{k\in I,k<i_m}(k-1)(c_{n+2-k}-c_k)\notag\\
                    &+\sum_{k\not\in I,2\le k<i_m}(k-1)(c_k-c_{n+2-k}) \biggr)+(i_m-1)\biggl( \sum_{k\in I,k>i_m}(n+2-2k)(c_k-c_{n+2-k})\notag\\
                    &+\sum_{k\not\in I,i_m<k\le u}(n+2-2k)(c_k-c_{n+2-k}) \biggr) \Biggr]
      \end{align}
   Let
      \begin{align*}
                   Z\left(i_m,n+2-i_m,\vec{C}^{(m-1)},\vec{C}^{(m)}\right)&=\frac{2}{n^2}\left(c_{n+2-i_m}-c_{i_m}\right)\\
                   &\hspace{-60pt}\Biggl( (n+2-2i_m)\sum_{k\in I,k<i_m}(k-1)(c_{n+2-k}-c_k)+(i_m-1) \sum_{k\in I,k>i_m}(n+2-2k)(c_k-c_{n+2-k})\Biggr)
      \end{align*}
   and
      \begin{align*}
                   R\left(i_m,n+2-i_m,\vec{C}^{(m-1)},\vec{C}^{(m)}\right)&=\frac{2}{n^2}\left(c_{n+2-i_m}-c_{i_m}\right)\\
                   &\hspace{-60pt}\Biggl( (n+2-2i_m)\sum_{k\not\in I,2\le k<i_m}(k-1)(c_k-c_{n+2-k})+(i_m-1) \sum_{k\not\in I,i_m<k\le u}(n+2-2k)(c_k-c_{n+2-k}) \Biggr)
      \end{align*}
    Note that, using the definitions, the whole increment of \,$f\left(\vec{C}^{(0)}\right)$ \,after the \textit{sum-`$n+2$' transform} from \,$\vec{C}^{(0)}$\, to\, $\vec{C}^{(p)}$\, (\,via $(i_1,n+2-i_1),(i_2,n+2-i_2)\dots (i_p,n+2-i_p)$\textit{-interchange}\,) is given by
      \begin{align}
      f\left(\vec{C}^{(p)}\right)-f\left(\vec{C}^{(0)}\right)&=\sum_{m=1}^{p}\left(f\left(\vec{C}^{(m)}\right)-f\left(\vec{C}^{(m-1)}\right) \right)\notag\\
                                                 &=\sum_{m=1}^{p}\left[Z\left(i_m,n+2-i_m,\vec{C}^{(m-1)},\vec{C}^{(m)}\right) +R\left(i_m,n+2-i_m,\vec{C}^{(m-1)},\vec{C}^{(m)}\right)\right]
      \end{align}
    We then prove the following identity
      \begin{align}
          \label{equ:zeron+2}
            \sum_{m=1}^{p}Z\left(i_m,n+2-i_m,\vec{C}^{(m-1)},\vec{C}^{(m)}\right)
                       &=\frac{2}{n^2}\Biggl(\sum_{m=1}^{p}(n+2-2i_m)(c_{n+2-i_m}-c_{i_m})\sum_{k\in I,k<i_m}(k-1)(c_{n+2-k}-c_k)\notag\\
                        &\hspace{30pt}-\sum_{m=1}^{p} (i_m-1)(c_{n+2-i_m}-c_{i_m}) \sum_{k\in I,k>i_m}(n+2-2k)(c_{n+2-k}-c_k) \Biggr)\notag\\
                       &\hspace{0pt}=\frac{2}{n^2}\Biggl(\sum_{m=2}^{p}\sum_{l=1}^{m-1}
                           (n+2-2i_m)(c_{n+2-i_m}-c_{i_m})(i_l-1)(c_{n+2-{i_l}}-c_{i_l})\notag\\
                        &\hspace{30pt}-\sum_{m=1}^{p-1}\sum_{l=m+1}^{p} (i_m-1)(c_{n+2-i_m}-c_{i_m})(n+2-2{i_l})(c_{n+2-{i_l}}-c_{i_l})\Biggr) \notag\\
                       &\hspace{0pt}=\frac{2}{n^2}\Biggl(\sum_{l=1}^{p-1}\sum_{m=l+1}^{p}
                           (i_l-1)(c_{n+2-{i_l}}-c_{i_l})(n+2-2i_m)(c_{n+2-i_m}-c_{i_m}) \notag\\
                        &\hspace{30pt}-\sum_{m=1}^{p-1}\sum_{l=m+1}^{p} (i_m-1)(c_{n+2-i_m}-c_{i_m})(n+2-2{i_l})(c_{n+2-{i_l}}-c_{i_l})\Biggr) \notag\\
                       &\hspace{0pt}=0
     \end{align}
   Therefore, the residual part
   \begin{align}
      \label{equ:diffn+2}
      f\left(\vec{C}^{(p)}\right)-f\left(\vec{C}^{(0)}\right)&=\sum_{m=1}^{p}R\left(i_m,n+2-i_m,\vec{C}^{(m-1)},\vec{C}^{(m)}\right)\notag\\
      &=\frac{2}{n^2}\Biggl(\sum_{m=1}^{p}
      (n+2-2i_m)(c_{n+2-i_m}-c_{i_m})\sum_{k\not\in I,2\le k<i_m}(k-1)(c_k-c_{n+2-k})\notag\\
      &\hspace{30pt}+\sum_{m=1}^{p} (i_m-1)(c_{n+2-i_m}-c_{i_m}) \sum_{k\not\in I,i_m<k\le u}(n+2-2k)(c_k-c_{n+2-k}) \Biggr)
   \end{align}
   The set \,$K=\left\{k\not\in I,2\le k<i_m,1\le m\le p \right\}\cup\left\{k\not\in I,i_m<k\le u,1\le m\le p \right\} $\, is not empty by definition. Otherwise, we will get a contradiction with condition (1) of definition \ref{def:sumn+2}.\par
   Considering the inequality \,$(c_{n+2-i_m}-c_{i_m})(c_k-c_{n+2-k})>0$ \,when\, $k\not\in I$ \,and\, $(n+2-2k)>0$ \,when\, $k\le u$, we finally obtain
      $$f\left(\vec{C}^{(p)}\right)-f\left(\vec{C}^{(0)}\right)>0$$
   Thus the proof is complete. \qed
   \end{proof}
   \noindent\textbf{Example 1}. Let us regard a sequence $C^{(0)}=[\,1,6,2,3,4,8,7,5\,]$, then $C^{(p)}=C^{(2)}=[\, 1,6,7,8,4,3,2,5\,]$ is its \textit{sum-`$n+2$' transform}. \,Note that the transform consists of $(3,7)$\textit{-interchange} and $(4,6)$\textit{-interchange}. So we get the set \,$I=\left\{i_1,i_2,\dots,i_p\mid 2\le i_1<i_2\dots <i_p\le u,c_{i_q}<c_{n+2-i_q}, q=1,2\dots p \right\}=\{i_1=3,i_2=4\}$.\par
   From (\ref{equ:diffn+2}), we know
     \begin{align*}
        &f\left(C^{(p)}\right)-f\left(C^{(0)}\right)=\sum_{m=1}^{2}R\left(i_m,n+2-i_m,C^{(m-1)},C^{(m)}\right)\\
        &=\frac{2}{n^2}(c_7-c_3)\biggl[(n+2-2i_1)\sum_{k\not\in I,2\le k<i_1}(k-1)(c_k-c_{n+2-k})
                               + (i_1-1) \sum_{k\not\in I,i_1<k\le u}(n+2-2k)(c_k-c_{n+2-k})\biggr] \\
         &+\frac{2}{n^2}(c_6-c_4)\biggl[(n+2-2i_2)\sum_{k\not\in I,2\le k<i_2}(k-1)(c_k-c_{n+2-k})
                               + (i_2-1) \sum_{k\not\in I,i_2<k\le u}(n+2-2k)(c_k-c_{n+2-k})\biggr] \\
                  &\hspace{0pt}=\frac{2}{n^2}(c_7-c_3)\left[(n+2-2i_1)(2-1)(c_2-c_8)\right]
                               +\frac{2}{n^2}(c_6-c_4)\left[(n+2-2i_2)(2-1)(c_2-c_8)\right] \\
                  &\hspace{0pt}=0.9375
      \end{align*}
   On the other hand, we know $f\left(C^{(0)}\right)=131.5$\, and\, $f\left(C^{(2)}\right)=132.4375$. It follows that
   $$ f\left(C^{(2)}\right)-f\left(C^{(0)}\right)=0.9375 $$
   This verifies the identity (\ref{equ:diffn+2}) as well as Theorem \ref{theo:sumn+2}.\par
    \begin{definition}
      \label{def:sumn+1}
        For sequence \,$\vec{C}=\{c_k,1\le k\le n\}$, do \,$(k,n+1-k)$\textit{-interchange} for each \,$k$\, satisfying \,$c_k>c_{n+1-k}$, $2\le k\le u^{\prime}$ \,until all \,$c_k<c_{n+1-k},2\le k\le u^{\prime}$. where \,$u^{\prime}=\lfloor n/2\rfloor$, $\lfloor \, \rfloor$ \,denotes the floor function. This is called the\, \textit{sum-`$n+1$' transform} of \,$\vec{C}$.
    \end{definition}

    \begin{theorem}
      \label{theo:sumn+1}
        A sum-`n+1' transform always results in a \textit{better} sequence.
    \end{theorem}

   \begin{proof}
   Without loss of generality, Let \,$I^{\prime}=\left\{i_1,i_2,\dots,i_{p^{\prime}}\mid 2\le i_1<i_2\dots <i_{p^{\prime}}\le u^{\prime} \right\}$ \,denote the aggregate set of the indexes \,$k$\, in \,$\vec{C}=\{c_k,1\le k\le n\}$\, satisfying\, $c_{i_q}>c_{n+1-i_q},\,q=1,2\dots p^{\prime}$, $u^{\prime}=\lfloor n/2\rfloor$.\par
   If \,$i+j=n+1$, from (\ref{equ:dfini}), we obtain
    \begin{align}
     \Delta{f}(i,n+1-i,\vec{C},\vec{C}^{\prime})&= \frac{2}{n^2}(n+1-2i)(c_{n+1-i}-c_i)
                       \Biggl[-\sum_{k=n+2-i}^{n}(n+1-k)c_k\notag \\
                       &\hspace{10pt}-(c_{n+1-i}+c_i)/{2}+\sum_{k=2}^{i-1}(k-1)c_k
                        +\frac{1}{n+1-2i}\sum_{k=i+1}^{n-i}\left( (1-2i)k-n-1+2i+ni\right)c_k\Biggr]\notag \\
                       &=\frac{2}{n^2}(n+1-2i)(c_{n+1-i}-c_i) \Biggl[
                        -\sum_{k=1}^{i-1}c_{n+1-k}- \frac{1}{n+1-2i}\sum_{k=i}^{n-i}(k-i)c_k \notag \\
                       &\hspace{125pt}-(c_{n+1-i}+c_i)/{2}+\sum_{k=2}^{i-1}(k-1)(c_k-c_{n+1-k}) \notag \\ &\hspace{125pt}+\frac{i-1}{n+1-2i}\sum_{k=i+1}^{u^{\prime}}(n+1-2k)(c_k-c_{n+1-k}) \Biggr]
    \end{align}

  We define \,$\tilde{\vec{C}}^{(m)}=\{ \tilde{c}_k^{(m)},k=1\dots n,m\le p^{\prime} \}$\, as the transform of sequence \,$\vec{C}$\,\,via\, $(i_1,n+1-i_1),(i_2,n+1-i_2)\dots (i_m,n+1-i_m)$\textit{-interchange}\,, which is also denoted as \,$\tilde{\vec{C}}^{(0)}$ \,for \,notational\, convenience. By definition, \,$\tilde{\vec{C}}^{(m)}$\, may also be obtained from \,$\tilde{\vec{C}}^{(m-1)}$\, by \,$(i_m,n+1-i_m)$\textit{-interchange}. The element \,$\tilde{c}_k^{(m)}$\, is given by
     \begin{align}
      \label{equ:ck_n+1}
        \tilde{c}_k^{(m)}&=
            \begin {cases} \tilde{c}_{n+1-k}^{(m-1)} &  k=i_m\\
                           \tilde{c}_{n+1-k}^{(m-1)} &  k=n+1-i_m\\
                               \tilde{c}_{k}^{(m-1)} &  others
            \end {cases}\notag \\
         &=
            \begin {cases} c_{n+1-k} &  k=i_q,1\le q\le m\\
                           c_{n+1-k} &  k=n+1-i_q,1\le q\le m\\
                               c_{k} &  others
            \end {cases}
     \end{align}
   Then
     \begin{align}
      f\left(\tilde{\vec{C}}^{(m)}\right)-f\left(\tilde{\vec{C}}^{(m-1)}\right)&=\Delta f\left(i_m,n+1-i_m,\tilde{\vec{C}}^{(m-1)},\tilde{\vec{C}}^{(m)}\right)\notag\\
           &=\frac{2}{n^2}(n+1-2i_m)\left(\tilde{c}_{n+1-i_m}^{(m-1)}-\tilde{c}_{i_m}^{(m-1)}\right) \Biggl[
               -\sum_{k=1}^{i_m-1}\tilde{c}_{n+1-k}^{(m-1)}\notag \\
               &\hspace{10pt}-\frac{1}{n+1-2i_m}\sum_{k=i_m}^{n-i_m}(k-i_m)\tilde{c}_k^{(m-1)}
                +\sum_{k=2}^{i_m-1}(k-1)\left(\tilde{c}_k^{(m-1)}-\tilde{c}_{n+1-k}^{(m-1)}\right)\notag\\
               &\hspace{10pt}-\left(\tilde{c}_{n+1-i_m}^{(m-1)}+\tilde{c}_{i_m}^{(m-1)}\right)\Bigl/{2}+\frac{i_m-1}{n+1-2i_m}\sum_{k=i_m+1}^{u^{\prime}}(n+1-2k)\left(\tilde{c}_k^{(m-1)}-\tilde{c}_{n+1-k}^{(m-1)}\right)\Biggr]\notag\\
           &=\frac{2}{n^2}(n+1-2i_m)\left(c_{n+1-i_m}-c_{i_m}\right) \Biggl[
               -\sum_{k=1}^{i_m-1}\tilde{c}_{n+1-k}^{(m-1)}-\frac{1}{n+1-2i_m}\sum_{k=i_m}^{n-i_m}(k-i_m)\tilde{c}_k^{(m-1)}\notag\\
            &\hspace{10pt}-\left(\tilde{c}_{n+1-i_m}^{(m-1)}+\tilde{c}_{i_m}^{(m-1)}\right)\Bigl/{2}+\sum_{k=2}^{i_m-1}(k-1)\left(\tilde{c}_k^{(m-1)}-\tilde{c}_{n+1-k}^{(m-1)}\right)\notag\\
            &\hspace{10pt}+\frac{i_m-1}{n+1-2i_m}\sum_{k=i_m+1}^{u^{\prime}}(n+1-2k)\left(c_k-c_{n+1-k}\right)\Biggr]
     \end{align}
   Note that for \,$i\in I^{\prime},k\not\in I^{\prime}$, we have
        \begin{gather*}
          (c_{n+1-i}-c_i)<0,\,\\
          (c_{n+1-i}-c_i)(c_k-c_{n+1-k})>0
        \end{gather*}
   It follows that
        \begin{align}
         \Delta f(i_m,n+1-i_m,\tilde{\vec{C}}^{(m-1)},&\tilde{\vec{C}}^{(m)})>\frac{2}{n^2}\left(c_{n+1-i_m}-c_{i_m}\right)\Biggl[(n+1-2i_m)
                      \sum_{k=2}^{i_m-1}(k-1)\left(\tilde{c}_k^{(m-1)}-\tilde{c}_{n+1-k}^{(m-1)}\right)\notag\\
                       &\hspace{110pt}+{(i_m-1)}\sum_{k=i_m+1}^{u^{\prime}}(n+1-2k)\left(c_k-c_{n+1-k}\right)\Biggr]\notag\\
            =&\frac{2}{n^2}\left(c_{n+1-i_m}-c_{i_m}\right) \Biggl[
                      (n+1-2i_m)\biggl(
                      \sum_{k\in I^{\prime},k<i_m}(k-1)\left(\tilde{c}_k^{(m-1)}-\tilde{c}_{n+1-k}^{(m-1)}\right)\notag\\
                      &\hspace{80pt}+\sum_{k\not\in I^{\prime},2\le k<i_m}(k-1)\left(c_k-c_{n+1-k}\right)\biggr)+(i_m-1)\notag\\
                       &\biggl(\sum_{k\not\in I^{\prime},i_m<k\le u^{\prime}}(n+1-2k)\left(c_k-c_{n+1-k}\right)
                       +\sum_{k\in I^{\prime},k>i_m}(n+1-2k)\left(c_k-c_{n+1-k}\right)
                               \biggr)
                                                            \Biggr]\notag\\
             >&\frac{2}{n^2}\left(c_{n+1-i_m}-c_{i_m}\right) \Biggl[
                      (n+1-2i_m)\sum_{k\in I^{\prime},k<i_m}(k-1)\left(\tilde{c}_k^{(m-1)}-\tilde{c}_{n+1-k}^{(m-1)}\right)\notag\\
                      &\hspace{80pt}+(i_m-1)\sum_{k\in I^{\prime},k>i_m}(n+1-2k)\left(c_k-c_{n+1-k}\right)
                                                              \Biggr]\notag\\
             =&\frac{2}{n^2}\left(c_{n+1-i_m}-c_{i_m}\right) \Biggl[
                      (n+1-2i_m)\sum_{k\in I^{\prime},k<i_m}(k-1)\left(c_{n+1-k}-c_k\right)\notag\\
                      &\hspace{80pt}+(i_m-1)\sum_{k\in I^{\prime},k>i_m}(n+1-2k)\left(c_k-c_{n+1-k}\right)
                                                             \Biggr]
        \end{align}
     Following derivations similar to (\ref{equ:zeron+2}) yield
     \begin{align*}
        \sum_{m=1}^{p^{\prime}}\left(c_{n+1-i_m}-c_{i_m}\right) &\Biggl[
                      (n+1-2i_m)\sum_{k\in I^{\prime},k<i_m}(k-1)\left(c_{n+1-k}-c_k\right)\\
                      &+(i_m-1)\sum_{k\in I^{\prime},k>i_m}(n+1-2k)\left(c_k-c_{n+1-k}\right)
                                                          \Biggr]=0
     \end{align*}
    So we may conclude that
      $$ f\left(\tilde{\vec{C}}^{(p^{\prime})}\right)-f\left(\tilde{\vec{C}}^{(0)}\right)=\sum_{m=1}^{p^{\prime}}\Delta f(i_m,n+1-i_m,\tilde{\vec{C}}^{(m-1)},\tilde{\vec{C}}^{(m)})>0 $$
    Where \,$\tilde{\vec{C}}^{(p^{\prime})}$\, is the \textit{sum-`$n+1$' transform} of \,$\tilde{\vec{C}}^{(0)}$, \,the conclusion follows. \qed
    \end{proof}

    \begin{theorem}
      \label{theo:optimal}
         If the optimal sequence \,$\vec{C}^{\ast}$\, is in the form of \,$c_1^{\ast}=a_1$\, and\, $c_n^{\ast}=a_2$, then the following hold.\par
          \renewcommand{\labelenumi}{\it{(\arabic{enumi})}}
          \begin{enumerate}
          \renewcommand{\baselinestretch}{1.3}\normalsize
           \item	Case n is even \,or \,$u=u^{\prime}: \,c_{n+2-k}^{\ast}<c_k^{\ast}<c_{n+1-k}^{\ast}\hspace{6pt} \textrm{for}\hspace{6pt} 2\le k\le u$;
           \item	Case n is odd \,or \,$u=u^{\prime}+1: \,c_k^{\ast}<c_{n+1-k}^{\ast}  \hspace{6pt}\textrm{for}\hspace{6pt}  2\le k\le u^{\prime} \hspace{6pt}\textrm{and}\hspace{6pt} c_{n+2-k}^{\ast}<c_k^{\ast} $\hspace{6pt}{for}\hspace{6pt}$ 2\le k\le u$.
         \end{enumerate}
    \end{theorem}

    \begin{proof}
    We \,will\, first\, prove \,$c_k^{\ast}>c_{n+2-k}^{\ast}$\, by\, contradiction. \,It \,is\, obvious\, that \,$c_2^{\ast}>c_{n}^{\ast}$\,, then\, the\, condition\, (1) \,\,of definition \ref{def:sumn+2} doesn't hold. If the condition (2) of definition \ref{def:sumn+2} is also not satisfied, that is, there must exist a set $I=\bigl\{i_1,i_2,\dots,i_p\mid 2\le i_1<i_2\dots <i_p\le u$, $c_{i_q}^{\ast}<c_{n+2-i_q}^{\ast}, q=1,2\dots p \bigr\}$.
    Now we will get a \,\textit{better}\, sequence by applying \,\textit{sum-`$n+2$' transform}. A contradiction.\par
    This implies that the condition (2) must hold. In other words, we obtain \,$c_k^{\ast}>c_{n+2-k}^{\ast}$ \,for\, $2\le k\le u$.\par
    Similarly, we have \,$c_k^{\ast}<c_{n+1-k}^{\ast}$ \,for\, $2\le k\le u^{\prime}$ .\par
    This completes the proof. \qed
    \end{proof}

    \begin{corollary}
       \label{cly:optimal}
           The optimal sequences are
           \begin{equation*}\vec{C}^{\ast}=
              \begin {cases} \{a_1,a_3,a_5,a_7,\dots a_n,a_{n-1},a_{n-3},\dots a_6,a_4,a_2 \} & n\hspace{5pt} is\hspace{5pt} odd\\
                             \{a_1,a_3,a_5,a_7,\dots a_{n-1},a_n,a_{n-2},\dots a_6,a_4,a_2 \} & n\hspace{5pt} is\hspace{5pt} even
              \end {cases}
           \end{equation*}
        and
           \begin{equation*}\vec{C}^{\ast d}=
              \begin {cases} \{a_1,a_2,a_4,a_6,\dots a_{n-1},a_n,a_{n-2},\dots a_7,a_5,a_3 \} & n\hspace{5pt} is\hspace{5pt} odd\\
                             \{a_1,a_2,a_4,a_6,\dots a_n,a_{n-1},a_{n-3},\dots a_7,a_5,a_3 \} & n\hspace{5pt} is\hspace{5pt} even
              \end {cases}
           \end{equation*}
    \end{corollary}

    \begin{proof}
    By Theorem \ref{theo:optimal} and Lemma \ref{lemma:kinds}, we readily obtain the single optimal solution \,$\vec{C}^{\ast}$\, in the form of \,$c_1^{\ast}=a_1$\, and\, $c_n^{\ast}=a_2$,
         \begin{equation*}\vec{C}^{\ast}=
              \begin {cases} \{a_1,a_3,a_5,a_7,\dots a_n,a_{n-1},a_{n-3},\dots a_6,a_4,a_2 \} & n\hspace{5pt} is\hspace{5pt} odd\\
                             \{a_1,a_3,a_5,a_7,\dots a_{n-1},a_n,a_{n-2},\dots a_6,a_4,a_2 \} & n\hspace{5pt} is\hspace{5pt} even
              \end {cases}
         \end{equation*}
    Then we immediately obtain another optimal solution of the form \,$c_1^{\ast d}=a_1$\, and\, $c_2^{\ast d}=a_2$\, by Lemma \ref{lemma:dual}
         \begin{equation*}\vec{C}^{\ast d}=
              \begin {cases} \{a_1,a_2,a_4,a_6,\dots a_{n-1},a_n,a_{n-2},\dots a_7,a_5,a_3 \} & n\hspace{5pt} is\hspace{5pt} odd\\
                             \{a_1,a_2,a_4,a_6,\dots a_n,a_{n-1},a_{n-3},\dots a_7,a_5,a_3 \} & n\hspace{5pt} is\hspace{5pt} even
              \end {cases}
         \end{equation*}
    We claim that the optimal solution of the form \,$c_1^{\ast d}=a_1$\, and\, $c_2^{\ast d}=a_2$\, is also unique. Otherwise, we may get more than one optimal solution of the form \,$c_1^{\ast}=a_1$ \,and \,$c_n^{\ast}=a_2$\, according to Lemma \ref{lemma:dual}, contradiction.\par
    Hence we obtain all the two optimal sequences. \qed
    \end{proof}
    \noindent\textbf{Remark.} The properties of the $\mathrm{VPS_{max}}$ problem can also be applied to constrain the solution of the CTV problems and obtain a solution closer to the optimal one.\par
   \noindent\textbf{Example 2.} Let us look at sequence $C=\{\,9,8,6,4,2,1,3,5,7\,\}$ and $C^{\prime}=\{\,9,8,5,3,2,1,4,6,7\,\}$. It is observed that both sequences are V-shaped \cite{Eilon-77} and the three largest elements are placed in the optimal positions (\,i.e. the largest element has to be placed in position 1 while the second and third largest elements should be placed in position 2 and $n$\,) \cite{Hall-91}. Then we could not determine which one is better in the sense of smaller variance based on existing theory. However, using Theorem \ref{theo:sumn+2}, we immediately see that $C^{\prime}$ is better with smaller variance for the CTV problems.
    
\section*{Appendix}
    The proof of Lemma \ref{lemma:dual} is provided in this section. Let $\alpha =2c_1+c_2+c_3+\dots +c_n$ and $\vec{S}^d=\{s_k^d,1\le k\le n\}$ denote the partial sum sequence of \,$\vec{C}^d$, Note that
        \vspace{-0pt}
        \begin{gather*}
              s_k=\sum_{m=1}^{k}c_m,s_k^{d}=c_1+\sum_{m=2}^{k}c_{n+2-m},2\le k\le n\\
              s_1=s_1^{d}=c_1,s_n=s_n^{d}\\
              s_{n+1-k}+s_k^{d}=s_k+s_{n+1-k}^{d}=\alpha\\
              \overline{\vec{S}}=c_1+\sum_{k=2}^{n}\frac{n+1-k}{n}c_k,\overline{\vec{S}^d}=c_1+\sum_{k=2}^{n}\frac{n+1-k}{n}c_{n+2-k}\\
              \overline{\vec{S}}+\overline{\vec{S}^d}=2c_1+\sum_{k=2}^{n}\frac{n+1-k}{n}c_k+\sum_{k=2}^{n}\frac{k-1}{n}c_{k}=\alpha
        \end{gather*}
    \vspace{-0pt}
    On the other hand
    \vspace{-0pt}
        \begin{gather*}
              s_{n+1-k}-s_{n+1-k}^{d}=s_k-s_k^{d}=\sum_{m=2}^{k}c_{m}-\sum_{m=n+2-k}^{n}c_{m}\\
              n(\overline{\vec{S}}-\overline{\vec{S}^d})=\sum_{k=2}^{n}(n+1-k)c_k-\sum_{k=2}^{n}(k-1)c_{k}=\sum_{k=2}^{n}(n+2-2k)c_k\\
              n(\overline{\vec{S}}^2-\overline{\vec{S}^d}^2)=\alpha \sum_{k=2}^{n}(n+2-2k)c_k
        \end{gather*}
    Therefore
       \begin{align*}
        \sum_{k=2}^{n-1}\left(s_k^2-(s_k^{d})^2\right)=&\sum_{k=2}^{n-1}\left[\left(s_k^2-(s_k^{d})^2\right)+\left(s_{n+1-k}^2-(s_{n+1-k}^{d})^2\right)\right]/2\\
                                           =&2\alpha \sum_{k=2}^{n}(s_k-{s_k^{d}})/2\\
                                           =&\alpha \sum_{k=2}^{n}\left(\sum_{m=2}^{k}c_{m}-\sum_{m=n+2-k}^{n}c_{m} \right)\\
                                           =&\alpha \sum_{k=2}^{n}(n+1-k)c_k-\alpha \sum_{k=2}^{n}(k-1)c_k\\
                                           =&n(\overline{\vec{S}}^2-\overline{\vec{S}^d}^2)
       \end{align*}
    We obtain
       \begin{align*}
       f(\vec{C})-f(\vec{C}^d)&=\left(\frac{1}{n}\sum_{k=1}^n s_k^2-\overline{\vec{S}}^2 \right)-\left(\frac{1}{n}\sum_{k=1}^n (s_k^{d})^2-\overline{\vec{S}^d}^2 \right)\\
       &=\frac{1}{n} \left( \sum_{k=2}^{n-1}\left(s_k^2 -(s_k^{d})^2\right) -n(\overline{\vec{S}}^2-\overline{\vec{S}^d}^2)  \right)=0
       \end{align*}
\bibliographystyle{spbasic}

\nocite{*}

\end{document}